\documentclass[a4paper,oneside,11pt]{article}

\usepackage{amsmath,amsfonts,amscd,amssymb}
\usepackage{longtable,geometry}
\usepackage[english]{babel}
\usepackage[utf8]{inputenc}
\usepackage[active]{srcltx}
\usepackage[T1]{fontenc}
\usepackage{graphicx}
\usepackage{pstricks}
\usepackage{bbm}
\usepackage{mdframed}
\usepackage{MnSymbol}
\usepackage{stmaryrd}
\usepackage{nicefrac}
\usepackage{tabularx}
\usepackage{multirow}
 \usepackage{rotating}
\usepackage{enumitem}
\usepackage{xcolor}
\usepackage{framed}

\colorlet{shadecolor}{blue!15}

\newtheorem{theorem}{Theorem}[section]

\newtheorem{corollary}[theorem]{Corollary}
\newtheorem{lemma}[theorem]{Lemma}
\newtheorem{proposition}[theorem]{Proposition}
\newtheorem{definition}[theorem]{Definition}

\newtheorem{remark}[theorem]{Remark}

\newcommand{\be}[1]{\begin{equation}\label{#1}}
\newcommand{\ee}{\end{equation}}
\numberwithin{equation}{section}

\newcommand{\ba}[1]{\begin{align}\label{#1}}
\newcommand{\ea}{\end{align}}
\numberwithin{equation}{section}

\newcommand{\ben}{\begin{equation*}}
\newcommand{\een}{\end{equation*}}
\numberwithin{equation}{section}

\newenvironment{proof}[1][\relax]
  {\paragraph{Proof\ifx#1\relax\else~of #1\fi}}%
  {~\hfill$\square$\par\bigskip}

\newcommand{\calC}{\mathcal{C}}

\newcommand{\calH}{\mathcal{H}}

\newcommand{\calV}{\mathcal{V}}

\newcommand{\bbE}{\mathbb{E}}

\newcommand{\bbN}{\mathbb{N}}

\newcommand{\bbP}{\mathbb{P}}

\newcommand{\bbR}{\mathbb{R}}

\newcommand{\bbZ}{\mathbb{Z}}

\newcommand{\ep}{\varepsilon}

\newcommand{\rk}[1]{\bgroup\color{red}%
  \par\medskip\hrule\smallskip%
  \noindent\textbf{#1}%
  \par\smallskip\hrule\medskip\egroup}

\newcommand{\bexo}{\begin{mdframed}[backgroundcolor=lightgray!20]
\scriptsize}
\newcommand{\eexo}{\end{mdframed}}
\title{Sharp threshold phenomena in statistical physics}

\author{Hugo Duminil-Copin\thanks{
\texttt{duminil@ihes.fr} Institut des Hautes \'Etudes Scientifiques and Universit\'e de Gen\`eve
 \newline
This research was funded by a IDEX Chair from Paris Saclay and by the NCCR SwissMap from the Swiss NSF. The author is thankful to anonymous referees for their numerous comments: they tremendously improved the paper.}}
\date{\today}

\newcommand{\ff}{{\mathbf f}}\newcommand{\TT}{{\mathbf T}}

\begin{document}
\maketitle

\begin{abstract}
This text describes the content of the Takagi lectures given by the author in Kyoto in 2017. The lectures present some aspects of the theory of sharp thresholds for boolean functions  and its application to the study of phase transitions in statistical physics.

 \end{abstract}

\section{Introduction}

In physics, a {\em phase transition}  is a discontinuous change of behavior in a physical system as some of its parameters (for instance the temperature, the density or the pressure) vary continuously. The most classical examples are probably the transitions between solid, liquid and gaseous states of matter, and the transition from a ferromagnet to a paramagnet at Curie's temperature, but physics offers many examples of phase transitions, whose understanding is crucial both theoretically and practically.

As an illustration for what will come next in these lectures, let us present the phase transition in a specific model for porous media called Bernoulli percolation. Consider the square lattice $\bbZ^2$ with vertex-set given by points of $\bbR^2$ with integer coordinates, and edge-set $\bbE^2$ given by pairs $\{x,y\}\subset \bbZ^2$ with $\|x-y\|=1$ ($\|\cdot\|$ is the Euclidean norm). Define the random graph, introduced by Broadbent and Hammersley in \cite{BroHam57}, obtained by keeping each edge of $\bbZ^2$ with probability $p$ (therefore erasing it with probability $1-p$) independently of the other edges. 
The large scale properties of the random graph constructed like that change drastically at a critical value of the parameter $p$. More precisely, there exists $p_c\in[0,1]$ such that 
\begin{itemize}[noitemsep]
\item If $p<p_c$, the probability that there is an infinite connected component is zero.
\item If $p>p_c$, the probability that there is an infinite connected component is one.
\end{itemize}
(Note that we do not say anything about the case $p=p_c$.)
This is an archetypal example of a phase transition in statistical physics: as the parameter $p$ (which can be interpreted as density) is varied continuously through the value $p_c$, the probability of  having an infinite connected component jumps from 0 to 1. 

It was conjectured early that the value of the critical point $p_c$ is equal to $1/2$, but proving this statement took more than twenty years. One of the goals of these lectures is to provide a modern proof of this statement, based on the notion of sharp threshold.
\begin{figure}[h]
\begin{center}\includegraphics[width=0.50\textwidth,angle=90]{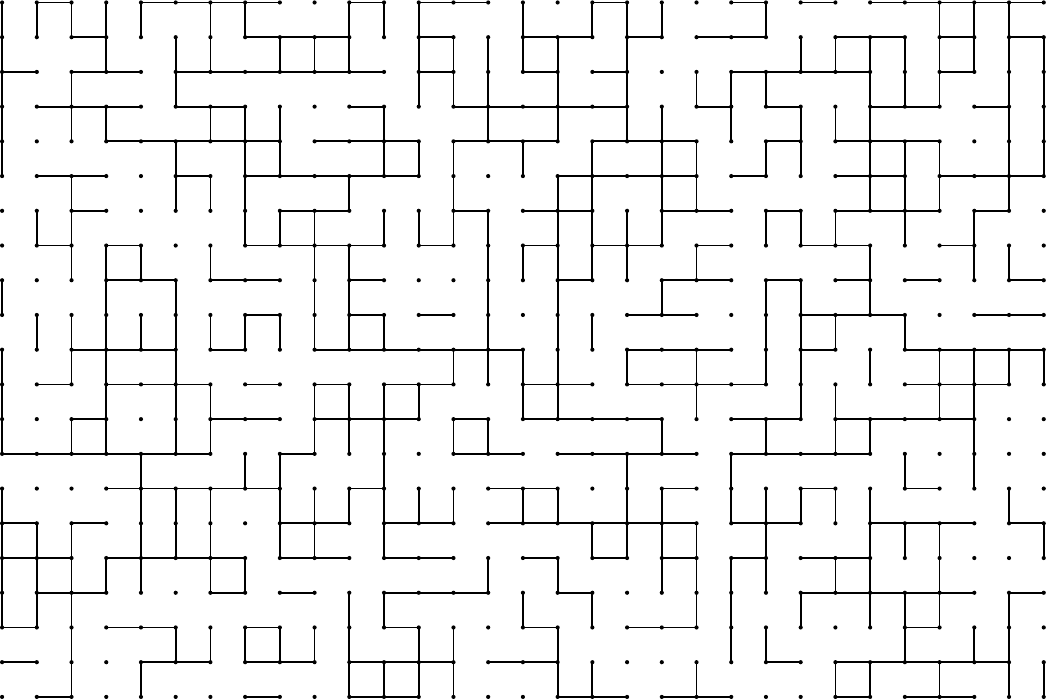}
\caption{A sampled configuration of Bernoulli percolation for the value of the parameter $p$ equal to $1/2$.}\end{center}
\end{figure}
\medbreak
In mathematics, a {\em finite} random system undergoes a {\em sharp threshold} if its qualitative behavior changes quickly as the result of a small perturbation of the parameters ruling the probabilistic structure. A fundamental example is provided by sharp thresholds undergone by the averages of boolean functions with respect to a product measure. More formally, 
for $p\in[0,1]$, let $\bbP_p$ be the law of a sequence of iid Bernoulli random variables with parameter $p$. In a slight abuse of notation, we will identify $\bbP_p$ with its restriction to the first $n$ random variables when considering a boolean function $\ff:\{0,1\}^n\rightarrow \{0,1\}$. An element of $\{0,1\}^n$ will generically be denoted by $\omega=(\omega_i:1\le i\le n)$. For future reference, we  denote the set $\{1,\dots,n\}$ by $[n]$. 
 We will also write $f(p):=\bbE_p[\ff]$ (and by extension $f_n(p):=\mathbb E_p[\ff_n]$ if we consider a sequence of boolean functions $\ff_n:\{0,1\}^n\rightarrow \{0,1\}$).

\begin{definition}
A sequence of increasing\footnote{With respect to the standard partial ordering on $\{0,1\}^n$: $\omega\le \omega'$ if  $\omega_i\le \omega'_i$ for every $1\le i\le n$.} boolean functions $(\ff_n)$ undergoes a {\em sharp threshold at $(p_n)$} if there exists $(\delta_n)$ tending to 0 such that $f_n(p_n-\delta_n)\rightarrow 0$ and $f_n(p_n+\delta_n)\rightarrow 1.$
\end{definition}
The notion of sharp threshold emerged in the combinatorics community studying graph properties of random graphs. In \cite{ErdRen59}, Erd\"os and Renyi introduced a model of random graph $G(n,p)$ with vertex-set $V=[n]$ and edge-set $E=\{i\in I:\omega_i=1\}$, where $I$ is the set of pairs of integers in $[n]$, which is identified with the set $[N]$ with $N=\binom n2$ to enter in our framework to define the measure as a product of Bernoulli.
The authors were originally interested in {\em graph properties} of $G(n,p)$, i.e.~properties of graphs that depend only on their isomorphism class\footnote{Isomorphism class is understood for the equivalence via graph isomorphisms. More formally, consider two graphs $G=(V,E)$ and $G'=(V',E')$, with $V$ and $V'$ the vertex-sets and $E$ and $E'$ the edge-sets. A map $T:V\longrightarrow V'$ is a graph isomorphism between $G$ and $G'$ if it is bijective and if $\{x,y\}\in E$ if and only if $\{T(x),T(y)\}\in E'$.}. They first focused on two specific properties, described below,

\paragraph{Example 1: Connectivity of the graph} Erd\"os and R\'enyi \cite{ErdRen59} proved that if $A_n$ is the event that the graph is connected, then $(\mathbbm1_{A_n})$ undergoes a sharp threshold at $p_n=\frac{\log n}{n}$.

\paragraph{Example 2: Existence of a giant connected component} Later, Erd\"os and R\'enyi also proved that if $B_n$ denotes the existence of a component in $G(n,p)$ of size larger or equal to $r_n$, where the sequence $(r_n)$ satisfies $\frac{r_n}{\log n}\rightarrow \infty$ and $\frac{r_n}{n}\rightarrow 0$, then $(\mathbbm1_{B_n})$ undergoes a sharp threshold at $p_n=\tfrac1 n$.
 \medbreak
Since this seminal paper, many other graph properties have been studied, including existence of certain induced subgraphs, etc. In fact, it was shown later (we will justify this in the next section) that many graph properties undergo sharp thresholds, namely the monotonic properties (examples include  being Hamiltonian, being non-planar, containing a clique of size $k$, having a diameter smaller than $r$, etc)

\paragraph{Example 3: Monotonic properties} Every sequence $(\mathbbm 1_{A_n})$ of increasing graph properties undergoes a sharp threshold with $\delta_n\log n\rightarrow\infty$.
\medbreak
Let us remark that properties not undergoing a sharp threshold must essentially depend on a (uniformly) bounded number of bits. 

When considering systems from statistical physics, the notion of sharp threshold often corresponds to a finite version of the  notion of phase transition. In these lectures, we propose to start from the mathematical theory of sharp thresholds and explain how this theory sheds some new light on the understanding of phase transitions in statistical physics.

 Section 3 presents two theorems from the study of boolean functions on product spaces. Section 4 discusses two applications to sharp phase transitions in statistical physics. Finally, Section 5 presents generalizations and applications of the theory of sharp thresholds for boolean functions to monotonic measures.

\section{How to prove that a sequence of boolean functions undergoes a sharp threshold?}\label{sec:2.2}

\subsection{The Margulis-Russo differentiation formula}

Let us start the study of averages of boolean functions by giving an expression for their derivatives. 
For a boolean function $\ff:\{0,1\}^n\rightarrow \{0,1\}$, introduce the notation $\nabla_i\ff(\omega):=f(\omega)- \ff(\mathsf{Flip}_i(\omega))$, where $\mathsf{Flip}_i(\omega)$ is the configuration obtained by flipping the state of $\omega_i$, i.e.~more formally
$$\mathsf{Flip}_i(\omega)_j:=\begin{cases}\ \ \, \omega_j&\text{ for }j\neq i,\\
1-\omega_j&\text{ for }j=i.\end{cases}$$
The {\em influence} of $1\le i\le n$:
$${\rm Inf}_i[\ff]:=\bbE_p[|\nabla_i\ff(\omega)|],$$

Note that the influence depends on $p$ even though we do not refer to it in the notation.

\begin{lemma}[Margulis \cite{Mar74}, Russo \cite{Rus81}]\label{lem:dif}
For $\ff:\{0,1\}^n\rightarrow\{0,1\}$ increasing, we have that $$f'(p)=\sum_{i=1}^n{\rm Inf}_i[\ff].$$
\end{lemma}

\begin{proof}Set $|\omega|=\sum_{i=1}^n \omega_i$.
Differentiating $f(p)=\sum_{\omega}\ff(\omega)p^{|\omega|}(1-p)^{n-|\omega|}$ with respect to $p$ immediately gives
  \begin{align}\nonumber f'(p)&=\tfrac1p\bbE_p[\ff(\omega)|\omega|]-\tfrac1{1-p}\bbE_p[\ff(\omega)(n-|\omega|)]\\
  &=\tfrac{1}{p(1-p)}\sum_{i=1}^n\bbE_p[\ff(\omega)(\omega_i-p)],\label{eq:ok}\end{align}
Note that conditioned on  $\omega\notin A_i=\{\omega:\nabla_i\ff(\omega)\ne 0\}$, then $\omega_i-p$ is (conditionally) independent of $\ff(\omega)$ and its (conditional) average in \eqref{eq:ok} is zero. Therefore, 
$$\bbE_p[\ff(\omega)(\omega_i-p)]=\bbE_p[\ff(\omega)(\omega_i-p)\mathbbm1_{A_i}].$$
For $\omega\in A_i$, the fact that $\ff$ is increasing implies that $\ff(\omega)=0$ if $\omega_i=0$, and $\ff(\omega)=1$ if $\omega_i=1$. We deduce that 
$$\bbE_p[\ff(\omega)(\omega_i-p)]=(1-p)\bbP_p[A_i\cap\{\omega_i=1\}].$$
Since $\omega_i$ and $A_i$ are independent, we conclude that
\begin{equation*}
\bbE_p[\ff(\omega)(\omega_i-p)]=p(1-p){\rm Inf}_i[\ff].
\end{equation*}
Inserting this expression in \eqref{eq:ok} implies the claim.
\end{proof}
The previous lemma immediately gives that $p\mapsto f(p)$ is increasing and differentiable. 
Therefore, if one can prove bounds of the type 
\begin{equation}\label{eq:fund}\sum_{i=1}^n{\rm Inf}_i[\ff]\ge C\,{\rm Var}_p(\ff)\end{equation}
for some large constant $C$, one deduces that the window of values of $p$ for which $f(p)$ remains far from $0$ and $1$ is necessarily small. Indeed, since $\ff$ takes values in $\{0,1\}$, we find 
$${\rm Var}_p(\ff)=f(p)(1-f(p))$$
which, together with Lemma~\ref{lem:dif}, enables one to rewrite \eqref{eq:fund} as
$$\Big(\log \tfrac{f(p)}{1-f(p)}\Big)'=\tfrac{f'(p)}{f(p)(1-f(p))}\ge C.$$
Let $p$ be such that $f(p)=\tfrac12$, then for any $\delta>0$, integrating the previous differential inequality between $p-\delta$ and $p$ gives
$f(p-\delta)\le e^{-C\delta}$. Similarly, integrating the differential inequality between $p$ and $p+\delta$ gives $ f(p+\delta)\ge 1-e^{-C\delta}$. In particular, this shows that for $C\delta\ge \log(\tfrac1{\ep})$, one has $f(p-\delta)\le\ep$ and $f(p+\delta)\ge 1-\ep$. The next sections are describing arguments leading to inequalities of the form of \eqref{eq:fund}.

\subsection{A sharp threshold inequality}

Historically, the general theory of sharp thresholds for discrete product spaces was initiated by Russo in \cite{Rus82} and Kahn, Kalai and Linial in \cite{KahKalLin88} in the case of the uniform measure on $\{0,1\}^n$, i.e.~to the case of $\bbP_p$ with  $p=1/2$. There, Kahn, Kalai and Linial used the Bonami-Beckner inequality \cite{Bek75,Bon70} to deduce inequalities between the variance of a boolean function and influences of this function (Beckner proved a similar inequality for Gaussian measures). Bourgain, Kahn, Kalai,
Katznelson and Linial \cite{BouKahKal92} extended these inequalities to product spaces $[0,1]^n$ endorsed with the uniform measure (for which the Bonami-Beckner inequality does not hold). Then, a discretization scheme enables one to deduce sharp threshold results on $\{0,1\}^n$ for $\bbP_p$ with arbitrary $p\in[0,1]$. Recently, the result was extended to any product space in \cite{GriJanNor15}. 
Here, we state a theorem due to Talagrand \cite{Tal94} which is essentially equivalent to the BKKKL result \cite{BouKahKal92} for what we have in mind. 

\begin{theorem}[Talagrand \cite{Tal94}]\label{maximum_influence}
There exists a constant $c>0$ such that for any $p\in[0,1]$ and $n\in\bbN$, the following holds. For any increasing boolean function $\ff:\{0,1\}^n\rightarrow\{0,1\}$,
$${\rm Var}_p(\ff)\le c\log\tfrac1{p(1-p)}\sum_{i=1}^n \frac{{\rm Inf}_i[\ff]}{\log (1/{\rm Inf}_i[\ff])}.$$

\end{theorem}
This statement is often used as follows: there must necessarily exist one influence which is larger than $c_p\tfrac{\log n}n{\rm Var}_p(\ff)$ (where $c_p=c\log\tfrac1{p(1-p)}$), which is immediate since there must be one $i$ with 
$$\frac{{\rm Inf}_i[\ff]}{\log (1/{\rm Inf}_i[\ff])}\ge\tfrac{c_p}n {\rm Var}_p(\ff)$$ 
(with $c_p$ maybe changed by a constant multiplicative factor).
\begin{proof}
The proof of this theorem is based on discrete Fourier analysis. 
Let us focus on the case $p=1/2$ which is slightly simpler due to the Bonami-Beckner inequality. We use the Fourier-Walsh expansion of $\ff$
$$\ff:=\sum_{S\subset [n]}\hat{\ff}(S) u_S,$$
where $u_S:=(-1)^{\sum_{i\in S}\omega_i}$ and $\hat{\ff}(S):=2^{-n}\sum_{\omega}\ff(\omega)u_S(\omega)$.
Observe that 
$$\widehat{\nabla_i\ff}(S)=\begin{cases}2\hat{\ff}(S)&\text{ if }i\in S,\\ \ \ \ 0&\text{ otherwise}.\end{cases}$$
Since $\hat\ff(\emptyset)=f(\tfrac12)$, and since by Parseval's inequality, $\bbE_{1/2}[\ff^2]=\sum_{S\subset[n]}\hat\ff(S)^2$, we deduce that 
$${\rm Var}(\ff)=\sum_{\substack{S\subset[n]\\ S\ne \emptyset}}\hat{\ff}(S)^2.$$We deduce that 
$${\rm Var}(\ff)=\sum_{i=1}^n\sum_{S\ni i}\frac{\hat{\ff}(S)^2 }{|S|}\le\sum_{i=1}^n\sum_{\substack{S\subset[n]\\ S\ne \emptyset}}\frac{\widehat{\nabla_i\ff}(S)^2 }{4|S|}\le \sum_{i=1}^n\sum_{\substack{S\subset[n]\\ S\ne \emptyset}}\frac{\widehat{\nabla_i\ff}(S)^2 }{|S|+1},$$
Writing $\tfrac1{|S|+1}=\int_0^1t^{|S|}dt$ and introducing for each $t\ge0$,
$$T_t\widehat{\nabla_i\ff}:=\sum_{S\subset[n]} t^{|S|}\widehat{\nabla_i\ff}(S)u_S,$$
we may write 
$${\rm Var}(\ff)\le \sum_{i=1}^n \int_0^1 \|T_t\widehat{\nabla_i\ff}\|_2^2dt,$$
where $\|\ff\|_\alpha^\alpha=2^{-n}\sum_\omega|\ff(\omega)|^\alpha$.
We now use the Bonami-Beckner inequality  \cite{Bek75,Bon70}, which states that 
$$\|T_t\widehat{\nabla_i\ff}\|_2\le \|\widehat{\nabla_i\ff}\|_{1+t^2}.$$ Since $\widehat{\nabla_i\ff}$ takes values in $\{-1,0,1\}$, we also have that $\|\widehat{\nabla_i\ff}\|_{1+t^2}={\rm Inf}_i[\ff]^{1/(1+t^2)}$. We deduce that 
\begin{equation}{\rm Var}(\ff)\le \sum_{i=1}^n {\rm Inf}_i[\ff]\cdot \int_0^1{\rm Inf}_i[\ff]^{\tfrac{1-t^2}{1+t^2}}dt.\label{eq:abcde}\end{equation}
Now, $\tfrac{1-t^2}{1+t^2}\ge 1-t$ and ${\rm Inf}_i[\ff]\le1$, so making the change of variables $s(t)=1-t$  gives
$$\int_0^1{\rm Inf}_i[\ff]^{\tfrac{1-t^2}{1+t^2}}dt\le\int_0^1{\rm Inf}_i[\ff]^{1-t}dt\le \int_0^1{\rm Inf}_i[\ff]^{s}ds\le \frac1{\log (1/{\rm Inf}_i[\ff])}.$$
Plugging this inequality in \eqref{eq:abcde} gives the result.
\end{proof}
The whole gain in the previous proof comes from the Bonami-Beckner inequality. This inequality is not obvious to prove, and we refer to \cite{Bek75,Bon70} for details.

Note that this inequality implies that $f'(p)$ is much larger than ${\rm Var}_p(\ff)$ as soon as all the influences are small (which can be seen counterintuitive since the derivative is the sum of the influences). More precisely, if all the influences are smaller than $\ep$, then $f'(p)$ is larger than $c_p\log(1/\ep){\rm Var}_p(\ff)$.

In general, it may be difficult to prove that all influences are small, but there is a particularly efficient way of using
Theorem~\ref{maximum_influence} when $A$ is invariant under a group acting transitively on $[n]$. 

\begin{theorem}\label{symmetric_sharp_threshold}
There exists $c>0$ such that for any $p\in[0,1]$ and $n\in\bbN$, the following holds. For an increasing boolean function $\ff:\{0,1\}^n\rightarrow\{0,1\}$ which is symmetric\footnote{Meaning that $f\circ\sigma=f$ for every $\sigma\in\mathfrak S$.} under a group $\mathfrak S$ acting transitively on $[n]$, 
$$f'(p)\geq c \log n\, {\rm Var}_p(\ff).$$ 
\end{theorem}

\begin{proof}
Since the boolean function is symmetric under a group $\mathfrak S$ acting transitively on $[n]$, we have that for each $i$ and $j$, $\ff=\ff\circ\sigma$ for any $\sigma\in\mathfrak S$ satisfying $\sigma(i)=j$. In particular, we deduce that 
$${\rm Inf}_i(\ff)={\rm Inf}_i(\ff\circ\sigma)={\rm Inf}_j(\ff),$$
and therefore all the influences are equal.  
Now, we are facing two cases:
\begin{itemize}
\item if ${\rm Inf}_i(\ff)\ge \frac{\log n}{n}$ for all $i$, then $$f'(p)=\sum_{i}{\rm Inf}_i(\ff)\ge \log n\ge \log n{\rm Var}_p(\ff).$$
\item if ${\rm Inf}_i(\ff)\le \frac{\log n}{n}$ for all $i$, then $\log(1/{\rm Inf}_i(\ff))\ge \log n-\log \log n$ for all $i$ and Theorem~\ref{maximum_influence} implies
$$f'(p)\ge c_p(\log n -\log \log n) {\rm Var}_p(\ff).$$
By modifying $c_p$ and choosing it small enough, we obtain the result.
\end{itemize}

\end{proof}
This theorem is already quite powerful since it guarantees that one may take $C=c\log n$. In particular, this theorem implies that every monotone graph property undergoes a sharp threshold (see Example 3 of the previous section). Indeed, by definition a graph property $A$ is invariant under graph isomorphism. In particular, it is invariant under relabeling of the vertices, and therefore $\mathbbm 1_A$ is invariant under a group $\mathfrak S$ acting transitively on the vertices of the graph.

\subsection{The O'Donnell-Schramm-Saks-Servedio inequality}

We now present another inequality enabling to derive bounds like \eqref{eq:fund}. This one is based on algorithms and was introduced to solve a conjecture of Yao \cite{yao1977probabilistic}.

Informally speaking, an algorithm associated with a boolean function $\ff$ takes $\omega\in\{0,1\}^n$ as an input, and reveals algorithmically the value of $\omega$ at different coordinates one by one until the value of $\ff(\omega)$ is determined. At each step, which coordinate will be revealed next depends on the values of $\omega$ revealed so far. The algorithm stops as soon as the value of $\ff$ is the same no matter the values of $\omega$ on the remaining coordinates. Then, the question is often to determine how many bits of information must be revealed before the algorithm stops (this quantity is sometimes referred to as the computational complexity of the boolean function).

Formally, an algorithm is defined as follows. 
For a $n$-tuple $x=(x_1,\dots,x_n)$ and $t\le n$, write $x_{[t]}=(x_1,\dots,x_t)$ and $\omega_{x_{[t]}}=(\omega_{x_1},\dots,\omega_{x_t})$. An {\em algorithm} $\TT=(i_1,\psi_t,t<n)$ takes $\omega\in\{0,1\}^n$ as an input and gives back an ordered sequence $(i_1,\dots,i_n)$ constructed inductively as follows: for any $2\le t\le n$,
$$i_t=\psi_t(i_{[t-1]},\omega_{i_{[t-1]}})\in [n]\setminus \{i_1,\dots,i_{t-1}\},$$
where $\psi_t$ is a function interpreted as the decision rule at time $t$ ($\psi_t$ takes the location and the value of the bits for the first $t-1$ steps of the induction, and decides the next bit to query). Note that the first coordinate $i_1$ is deterministic.
For $\ff:\{0,1\}^n\rightarrow \bbR$, define 
\begin{equation*}\label{eq:ddd}\tau(\omega)=\tau_{\ff,\bf T}(\omega):=\min\big\{t\ge1:\forall x\in\{0,1\}^E,\quad x_{i_{[t]}}=\omega_{i_{[t]}}\Longrightarrow f(x)=f(\omega)\big\}.\end{equation*}
\begin{remark}In computer science, an algorithm is usually associated directly to a boolean function $\ff$ and defined as a rooted directed tree in which each internal nodes are labeled by elements of $[n]$, leaves by possible outputs of $\ff(\omega)$, and edges are in correspondence with the possible values of the bits at vertices (see \cite{OSSS} for a formal definition). In particular, the algorithms are usually defined up to $\tau$, and not later on. \end{remark}

The OSSS inequality, originally introduced by O'Donnell, Saks, Schramm and Servedio in \cite{OSSS} as a step toward a conjecture of Yao \cite{yao1977probabilistic}, relates the variance of a boolean function to the influence and the computational complexity of an algorithm for this function. 
\begin{theorem}[OSSS inequality \cite{OSSS}]
\label{thm:OSSS}
Consider $p\in[0,1]$ and $n\in\bbN$. Fix an increasing boolean function $\ff:\{0,1\}^n\longrightarrow \{0,1\}$ and an algorithm $\TT$. We have
\begin{equation}
    \label{eq:OSSS}
 \mathrm{Var}_p(\ff)~\le~   p(1-p) \sum_{i=1}^n  \delta_i(\TT) \, {\rm Inf}_i(\ff),
  \end{equation}
  where 
$
\delta_i(\TT)=\delta_i(\ff,\TT):=\bbP_p\big[\exists t\le \tau(\omega)\::\:i_t=i\big]
$
is called the {\em revealment} of $\ff$ for the algorithm $\TT$ and the bit $i$.
\end{theorem}

\newcommand{\e}{\mathbf i}
\begin{proof}
Consider two independent sequences $\omega$ and $\tilde\omega$ of iid Bernoulli random variables of parameter $p$. Write $\mathbb P$ for the joint measure of these variables (and $\mathbb E$ for its expectation). 
Construct $\e$ by setting $\e_1=i_1$ and for $t\ge1$,  
$\mathbf{i}_{t+1}:=\psi_t(\mathbf{i}_{[t]},\omega_{\mathbf{i}_{[t]}})$. Note that the construction of $\mathbf i$ relies solely on $\omega$ and does not involve $\tilde\omega$. Define $$\tau:=\min\{t\ge1:\forall x\in\{0,1\}^E,x_{\mathbf{i}_{[t]}}=\omega_{\mathbf{i}_{[t]}}\Rightarrow \ff(x)=\ff(\omega)\}.$$ Finally, for $0\le t\le n$, define $$\omega^t:=(\tilde\omega_{\mathbf{i}_{1}},\dots,\tilde\omega_{\mathbf{i}_{t}},\omega_{\mathbf{i}_{t+1}},\dots,\omega_{\mathbf{i}_{\tau-1}},\tilde\omega_{\mathbf{i}_{\tau}},\tilde\omega_{\mathbf{i}_{\tau+1}},\dots,\tilde\omega_{\mathbf{i}_{n}}),$$where it is understood that the $n$-tuple under parentheses is equal to $\tilde\omega$ if $t\ge \tau$. 

Now, observe that $\ff$ takes values in $\{0,1\}$, therefore
$${\rm Var}_p(\ff)=\bbE_p[(\ff-f(p))^2]\le \tfrac12\bbE_p[|\ff-f(p)|].$$ 
Since $\omega^0$ and $\omega$ coincide on $\mathbf i_{[\tau]}$, we deduce that $\ff(\omega^{0})=\ff(\omega)$. Also, $\omega^n=\tilde\omega$ so that $\ff(\omega^{n})=\ff(\tilde\omega)$. As a consequence,  
conditioning on $\omega$ gives
\begin{equation*}\label{eq:a}2{\rm Var}_p(\ff)\le \bbE_p[|\ff-f(p)|]=\bbE\Big[\Big|\,\bbE[\ff(\omega^{0})|\omega]-\bbE[\ff(\omega^{n})|\omega]\,\Big|\Big]\le\mathbb E[|\ff(\omega^{0})-\ff(\omega^{n})|].\end{equation*}
Since $\omega^t=\omega^{t-1}$ for any $t>\tau$, the right-hand side is smaller than or equal to
\begin{align*}\sum_{t=1}^{n}\mathbb E[|\ff(\omega^t)-\ff(\omega^{t-1})|]&=\sum_{i=1}^n \sum_{t=1}^{n}\mathbb E\Big[\mathbb E\big[\,|\ff(\omega^t)-\ff(\omega^{t-1})|\,\big|\,\omega_{\mathbf{i}_{[t-1]}}\big]\,\mathbbm1_{t\le \tau,\mathbf{i}_t=i}\Big].
\end{align*}
We now use the key property of the construction of the $\omega^t$. Conditionally on $\omega_{\mathbf{i}_{[t-1]}}$ and $\{t\le \tau,\mathbf{i}_t=i\}$, both $\omega^t$ and $\omega^{t-1}$ are independent sequences of iid Bernoulli random variables since both involve only $\tilde\omega$ on edges in $\mathbf{i}_{[t-1]}$. Furthermore, they differ (possibly) at $i$ since $\omega^t_i=\tilde\omega_i$ and $\omega^{t-1}_i=\omega_i$. We insist on the fact that this is the fundamental property that we were looking for when defining $\omega^t$. We deduce that 
\begin{equation*}\mathbb E\big[\,|\ff(\omega^t)-\ff(\omega^{t-1})|\,\big|\,\omega_{\mathbf{i}_{[t-1]}}\big]~=2p(1-p)\mathbb E_p[|\nabla_i\ff(\omega)|]=2p(1-p){\rm Inf}_i[\ff].\end{equation*}
Recalling that $\sum_{t=1}^{n}\mathbb P[t\le \tau,\mathbf{i}_t=i]=\delta_i(\TT)$ concludes the proof.
\end{proof}

\section{Applications to Bernoulli percolation on $\bbZ^d$}

We now focus on two applications to Bernoulli percolation. Consider the $d$-dimensional lattice with vertex set $\bbZ^d$ and edge set $\bbE^d$ given by pairs $\{x,y\}\subset\bbZ^d$ with $\|x-y\|=1$. We do not work on boolean functions defined on $\{0,1\}^n$ anymore but rather on $\{0,1\}^E$ with $E\subset \bbE^d$ being a finite set. In particular, we will use the notation $e$ instead of $i$ to refer to elements of $E$ (which are all edges of $\bbZ^d$).  Note that the theorems proved in the previous section are also valid in this context.

Set $\Lambda_n=[-n,n]^d$ and $\partial\Lambda_n:=\Lambda_n\setminus\Lambda_{n-1}$. Also, set $X\leftrightarrow Y$ if there exists a path in $\omega$ from $X$ to $Y$. Finally, we write $0\leftrightarrow\infty$ for the event that 0 is in an infinite connected component.

\subsection{Critical point of Bernoulli percolation on $\bbZ^2$}
In this section, we discuss the proof of the following theorem.
\begin{theorem}[Kesten \cite{Kes80}]\label{thm:main}
The critical point of Bernoulli percolation on the square lattice is equal to $1/2$.
\end{theorem}

We present a method initiated first by Russo \cite{Rus82}. It was later developed further by Bollob\`as and Riordan \cite{BolRio06,BolRio06b,BolRioc}. It is based on the existence of a sharp threshold for so-called crossing probabilities. For two integers $n$ and $m$, define the rectangle $R(n,m):=[0,n]\times[0,m]$. Consider the event $\calH(n,m)$  to be events that the configuration $\omega$ contains a path in $R(n,m)$ from the left side to the right side\footnote{The left side is $\{0\}\times[0,m]$ and the right side $\{n\}\times[0,m]$. We take this opportunity to also define the bottom side $[0,n]\times\{0\}$ and the top side $[0,n]\times\{m\}$ for future reference.} of $R(n,m)$. In this case, we say that $R(n,m)$ is crossed horizontally. Similarly, one defines $\calV(n,m)$ to be the event that the configuration $\omega$ contains a path in $R(n,m)$ from the bottom to the top of $R(n,m)$. In this case, we say that $R(n,m)$ is crossed vertically.

Let us start by a simple observation.
\begin{proposition}
We have $\bbP_{1/2}[\calH(n-1,n)]=\tfrac12$ for all $n$.
\end{proposition}

\begin{proof}
Consider the dual lattice $(\bbZ^2)^*:=(\tfrac12,\tfrac12)+\bbZ^2$ of the lattice $\bbZ^2$ defined by putting a vertex in the middle of each face, and edges between nearest neighbors. Each edge $e\in\bbE^2$ is in direct correspondence with an edge $e^*$ of the dual lattice crossing it in its middle. For a finite graph $G=(V,E)$, let $G^*$ be the graph with edge-set $E^*=\{e^*,e\in E\}$ and vertex-set given by the endpoints of the edges in $E^*$. 

A configuration $\omega$ in $\{0,1\}^E$ is naturally associated with a dual configuration $\omega^*$ on $\{0,1\}^{E^*}$ as follows: for every $e\in E$, set $\omega^*_{e^*}:=1-\omega_e$.
Note that if the law of $\omega$ is a product of independent Bernoulli variables with parameter $p$, then the law of $\omega^*$ is a product of Bernoulli variables with parameter $1-p$.

Observe that the complement of the event $\calH(n-1,n)$ is the event that there exists a path of edges in $\omega^*$ going from top to bottom in the graph $R(n-1,n)^*$. Using the symmetry by rotation by $\pi/2$, one sees that at $p=1/2$, these two events have the same probability, which must therefore be equal, for every $n\ge1$, so that
\begin{equation}\label{eq:1/2}\bbP_{1/2}[\calH(n-1,n)]=\tfrac12.\end{equation}
\end{proof}

In particular, crossing probabilities for squares (they are not quite squares but it is pretty much the same) do not tend to 0 or 1 as $n$ tends to infinity. One may wonder whether this is simply due to the fact that we chose a rectangle which is almost a square, or whether this holds for every rectangle which is not too degenerate, meaning that they are not too flat. 
We are going to see that this is the case. This property, which is called the Box-Crossing Property, is absolutely fundamental for the understanding of the phase $p=1/2$.
\begin{theorem}\label{thm:RSW}
For any $\rho>0$, there exists $c=c(\rho)>0$ such that for all $n\ge1$,
$$c\le \bbP_{1/2}[\calH(\rho n,n)]\le 1-c.$$
\end{theorem}
Note that we immediately deduce a similar bound for probabilities of being crossed vertically.
The uniform upper bound follows easily from the uniform lower bound and duality since the complement of the event that a rectangle is crossed vertically is the event that the dual rectangle is crossed horizontally in the dual configuration. 

Also, as soon as we have to our disposal a uniform lower bound (in $n$) for some $\rho=1+\ep>1$, then one can easily combine crossings in different rectangles to obtain a uniform lower bound for any $\rho'>1$. Indeed, define (for every integer $i\ge0$) the rectangles $R_i:=[i\ep n,(i\ep+\rho)n]\times[0,n]$ and the squares $S_i:=R_i\cap R_{i+1}$. Also define $\calH(R_i)$ and $\calV(S_i)$ to be the events that $R_i$ is crossed horizontally, and $S_i$ vertically. Then, 
\begin{align}\label{eq:u}
\bbP_{1/2}[\calH(\rho' n,n)]\ge  \bbP_{1/2}\Big[\bigcap_{ i=0}^{\lceil (\rho'-1)/\ep\rceil}(\calH(R_i)\cap\calV(S_i))\Big]\stackrel{\rm (FKG)}\ge c(\rho)^{2\lceil \rho'/\ep\rceil}.
\end{align}
Above, we used the following inequality, known as the Harris or FKG inequality (see \cite{Gri99a}): for any two increasing boolean functions $\ff$ and $\mathbf g$, 
\begin{equation}\label{eq:FKG}\bbE_p[\ff \mathbf{g}]\ge\bbE_p[\ff]\bbE_p[\mathbf g].\end{equation}
Note that for the event under consideration in \eqref{eq:u}, indication functions were increasing.

Unfortunately, we cannot start a priori from an estimate with $\rho>1$ and must deal with the case $\rho=1$.  This will in fact be the major obstacle: the main difficulty of Theorem~\ref{thm:RSW} lies in passing from crossing squares with probabilities bounded uniformly from below to crossing rectangles in the hard direction with probabilities bounded uniformly from below. A statement claiming that crossing a rectangle in the hard direction can be expressed in terms of the probability of crossing squares is called a Russo-Seymour-Welsh type theorem. For Bernoulli percolation on the square lattice, such a result was first proved in \cite{Rus81,SeyWel78}.
Since then, many proofs have been produced, among which \cite{BolRio06,BolRio06c,BolRio10,Tas14b,Tas14}.
This seemingly technical statement is in fact at the root of virtually every study of the critical phase of Bernoulli percolation.

\begin{proof}
As mentioned before the proof, it is sufficient to prove that the crossing probability in the hard direction, for a rectangle with $\rho=3/2$:
$$\bbP_{1/2}[\calV(2n,3n)]\ge \tfrac1{128}.$$
We choose to work with vertical crossings of the rectangle $R:=[-n,n]\times[-n,2n]$. We will need some additional notation.
Set $S:=[0,n]^2$ and $S':=[-n,n]^2$. Also, define $\ell:=[-n,n]\times\{-n\}$ to be the bottom side of $R$ (or equivalently of $S'$).

Let $A$ (resp.~$A'$) be the event that there exists a bottom-top (left-right) crossing of $S$, and $B$ be the event that there exists a left-right crossing of $S$ that is connected to $\ell$ in $S'$.
For a path $\gamma$ from left to right in $S$, and $\sigma(\gamma)$ the reflection of this path with respect to $\{0\}\times\bbZ$, define the set $V(\gamma)$ of vertices in $S'$ below $\gamma\cup\sigma(\gamma)$ (see Fig.~\ref{fig:first RSW} on the left).
Now, on $A'$, condition on the highest left-right crossing $\Gamma$ of $S$. We find that
\begin{align*}\bbP_{1/2}[B]&\ge \sum_{\gamma}\bbP_{1/2}[B\,|\,A'\cap\{\Gamma=\gamma\}]\,\bbP_{1/2}[\{\Gamma=\gamma\}\cap A']\\
&\ge\sum_{\gamma}\bbP_{1/2}[\gamma\leftrightarrow\ell\text{ in }V(\gamma)]\,\bbP_{1/2}[\{\Gamma=\gamma\}\cap A']\\&\ge\tfrac14\sum_{\gamma}\bbP_{1/2}[\{\Gamma=\gamma\}\cap A']=\tfrac14\bbP_{1/2}[A']\ge\tfrac18.\end{align*}
In the third line, to deduce the lower bound $1/4$, we used the facts that conditioned on $A\cap\{\Gamma=\gamma\}$, the configuration in $V(\gamma)$ is a Bernoulli percolation of parameter $1/2$ (since $A\cap\{\Gamma=\gamma\}$ is measurable with respect to the random variables $\omega_e$ for edges $e$ on $\gamma$ or above $\gamma$), the symmetry and the fact that the probability of a bottom-top crossing in $V(\gamma)$ is larger than $1/2$ (since it is easier than a bottom-top crossing of $S'$). 
Fig.~\ref{fig:first RSW} on the right illustrates that $R$ is crossed vertically if the three events $A$, $B$ and $\widetilde B$ occur, where $\widetilde B$ is the event that there exists a left-right crossing of $S$ which is connected to $[-n,n]\times\{2n\}$ in $[-n,n]\times[0,2n]$. By symmetry,
$$\bbP_{1/2}[\widetilde B]=\bbP_{1/2}[B]\ge\tfrac18.$$
The FKG inequality \eqref{eq:FKG} (used in the second inequality) implies that
\begin{align*}\bbP_{1/2}[\calV(2n,3n)]&\ge\bbP_{1/2}[A\cap B\cap \widetilde B]\\
&\ge \bbP_{1/2}[A]\bbP_{1/2}[B]\bbP_{1/2}[\widetilde B]\ge \tfrac1{128}.\end{align*}
\end{proof}
\begin{figure}\centerline{
\includegraphics[width=0.40\textwidth]{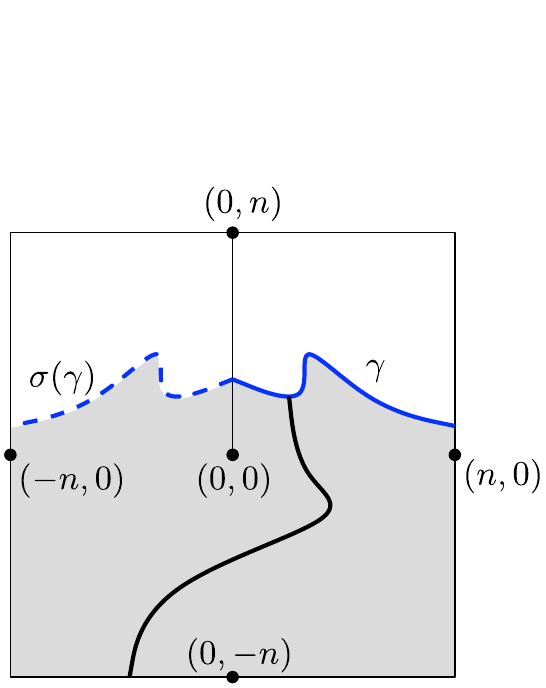}\quad\quad\includegraphics[width=0.40\textwidth]{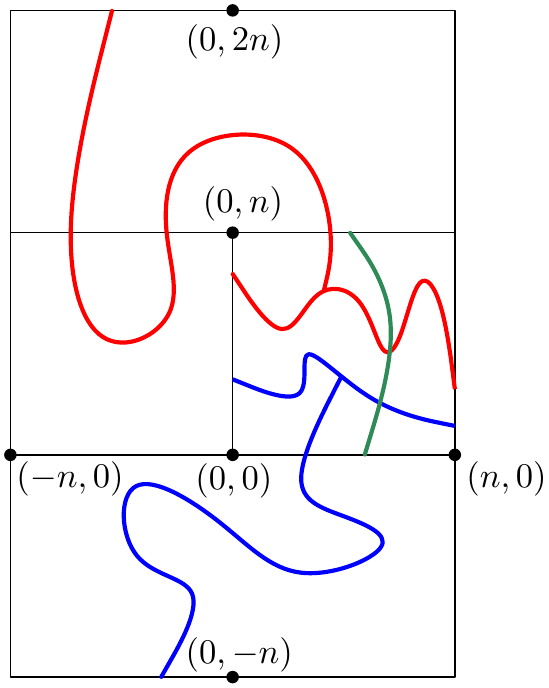}}
\caption{\textbf{Left.} The set $V(\gamma)$. {\bf Right.} The combination of the events $A$, $B$ and $\widetilde B$ imply the event that $R$ is crossed vertically.}\label{fig:first RSW}
\end{figure}

Note that a trivial corollary of the previous statement is the following. 
\begin{corollary}\label{cor:aa}
There exists $\alpha>0$ such that for all $n\ge1$,
$\bbP_{1/2}[0\leftrightarrow \partial\Lambda_n]\le n^{-\alpha}.$
In particular, $p_c\ge1/2$.
\end{corollary}

\begin{proof}
Consider the event that $\Lambda_k$ is connected to $\partial\Lambda_{2k}$. For this event to happen, it is necessary that one of the four rotated versions of the event that $[-2k,2k]\times[-k,k]$ is crossed vertically must occur. Therefore, the FKG inequality applied to the complements $A_1,\dots,A_4$ of these events implies that 
$$\bbP_{1/2}[\Lambda_k\leftrightarrow \partial\Lambda_{2k}]\le1-\bbP_{1/2}[A_1\cap\dots\cap A_4]\le 1-\bbP_{1/2}[A_1]^4\le 1-c^4=:c_1<1.$$
Since $0\leftrightarrow \partial\Lambda_n$ is included in the intersection of the events that $\Lambda_k\leftrightarrow\partial\Lambda_{2k}$, where $k\le n$ is a power of 2, the independence implies that 
$$\bbP_{1/2}[0\longleftrightarrow \partial\Lambda_n]\le c_1^{\lfloor\log_2(n)\rfloor}\le n^{-\alpha}$$
provided $\alpha$ is chosen small enough.

To prove that $p_c\ge1/2$, observe that by letting $n$ tends to infinity, we find that $\bbP_p[0\leftrightarrow\infty]=0$. Using that there are countably many vertices in $\bbZ^2$, and that for each one of them, the probability of being connected to infinity is zero (by invariance under translations), we deduce that the probability that there is an infinite connected component in $\omega$ is 0.
\end{proof}

Now that we proved that crossing probabilities remain bounded away from 0 and 1 at $p=1/2$, it is natural to ask oneself whether this is also the case for the values of $p$ that are not equal to $1/2$. This is where we will use Theorem~\ref{maximum_influence}. We will prove the following statement:

\begin{proposition}
For any $p>1/2$, there exists $\beta=\beta(p)>0$ such that 
$\bbP_p[\calH(2n,n)]\ge 1-\tfrac1{\beta}n^{-\beta}$.
\end{proposition}

\begin{proof}
Consider the boolean function $\ff:=\mathbbm 1_{\calH(2n,n)}$. Fix an edge $e$ of $R(2n,n)$ and observe that if $\nabla_e\ff(\omega)\ne 0$, then one of the endpoints of the dual edge $e^*$ of $e$ must be connected by a path in the dual configuration $\omega^*$ of $\omega$ to distance $n/2$. Since $\omega^*$ is sampled according to iid Bernoulli random variables of parameter $1-p$, Corollary~\ref{cor:aa} implies that
$${\rm Inf}_e(\ff)\le 2\bbP_{1-p}[0\leftrightarrow \partial\Lambda_{n/2}]\le 2\bbP_{1/2}[0\leftrightarrow \partial\Lambda_{n/2}]\le \tfrac1N,$$
where $N=\tfrac12(\tfrac n2)^\alpha$.
As a consequence, we deduce from Theorem~\ref{symmetric_sharp_threshold} that for any $p>1/2$,
$$f'(p)\ge c\log (N) {\rm Var}_p(\ff).$$
Integrating this differential inequality between $1/2$ and $p$ gives that 
$$f(p)\ge 1-\tfrac1{f(1/2)}N^{-c(p-1/2)}.$$
The result follows by setting $\beta$ small enough.
\end{proof}

\begin{proof}[Theorem~\ref{thm:main}]
We already know that $p_c\ge 1/2$. Let us prove the other inequality by proving that for $p>1/2$, the probability that there exists an infinite connected component in $\omega$ is 1. Let $A_n$ and $B_n$ be the events of $\calH(2^{n+1},2^n)$ and $\calV(2^n,2^{n+1})$ respectively. 
Observe that if $A_n$ and $B_n$ occur for all but finitely many $n$, then there exists an infinite connected component in $\omega$.

The previous proposition implies that
$$\sum_{n=1}^\infty \bbP_p[A_n^c]\le \tfrac1{\beta}\sum_{n=1}^\infty 2^{-\beta n}$$
so that the Borel-Cantelli lemma\footnote{The Borel-Cantelli lemma states that if $(A_n)$ is a sequence of events such that $\sum_{n=1}^\infty \bbP[A_n]<\infty$, then the probability that there are infinite many $n$ such that $A_n$ occurs is zero.} implies that the probability that $A_n$ occurs for all but finitely many $n$ is 1. By symmetry by rotation by an angle of $\pi/2$, we immediately deduce the same for the events $B_n$. In conclusion, we proved that the probability that there exists an infinite connected component in $\omega$ is 1.
\end{proof}

\subsection{Sharpness of the phase transition for Bernoulli percolation on $\bbZ^d$}
In higher dimensions, it is hopeless to try to compute the exact value of the critical point (one does not expect it to be equal to any nice number, for instance rational or even algebraic). Nonetheless, one can still try to prove that the model undergoes a sharp phase transition, meaning that probabilities to be connected to distance $n$ decay very fast when $p<p_c$. 
 \begin{theorem}\label{thm:perco} Consider Bernoulli percolation on $\bbZ^d$,
\begin{enumerate}[noitemsep]
\item\label{item:1}For $p<p_c$,  there exists $c_p>0$ such that for all $n\ge1$,
$\bbP_p[0\leftrightarrow \partial\Lambda_n]\le \exp(-c_pn)$.
\item\label{item:2} There exists $c>0$ such that for $p>p_c$, $\bbP_p[0\leftrightarrow \infty]\ge
 c (p-p_c)$.
\end{enumerate}
\end{theorem}
Note that the theorem does not mention anything on the $p=p_c$ phase. The reason is that the proof relies deeply on moving the value of $p$. The question of the $p=p_c$ phase is tremendously difficult in general and we avoid discussing it here. Let us also mention that the second item is often called the mean field lower bound. The lower bound is matched (up to constant) for $d\ge11$ \cite{FitHof17}, but is expected not to be sharp for small values of $d$ (this fact is known in dimension 2).

Theorem~\ref{thm:perco} was first  proved by Aizenman,
  Barsky~\cite{AizBar87} and Menshikov~\cite{Men86} (these two proofs are presented in \cite{Gri99a}). See also a recent short proof  \cite{DumTas15,DumTas15a}. Here, we choose to present a new proof \cite{DumRaoTas17} using the OSSS inequality.
Let us start the proof with a general lemma which is nothing but an undergrad exercise in analysis.   
  \begin{lemma}\label{lem:technical}
Consider a converging sequence of differentiable functions $f_n:[0,x_0]\longrightarrow [0,M]$ which are increasing in $x$ and satisfy
 \begin{equation}\label{eq:mlem}f_n'\ge \frac{n}{\Sigma_{n}}f_n\end{equation}
 for all $n\ge1$, where $\Sigma_n=\sum_{k=0}^{n-1}f_k$. Then, there exists $x_1\in[0,x_0]$ such that  
 \begin{itemize}[noitemsep]
 \item[{\bf P1}] For any $x<x_1$, there exists $c_x>0$ such that for any $n$ large enough,
 $f_n(x)\le \exp(-c_x n).$
  \item[{\bf P2}] For any $x>x_1$, $\displaystyle f=\lim_{n\rightarrow \infty}f_n$ satisfies $f(x)\ge x-x_1.$
 \end{itemize}
  \end{lemma}
  
  \begin{proof}
Define 
$$ x_1 := \inf \{ x \, : \, \limsup_{n \rightarrow \infty} \frac{\log \Sigma_n(x)}{\log n} \geq 1 \}. $$

\paragraph{Assume $x<x_1$.} Fix $\delta>0$ and set $x'=x-\delta$ and $x''=x-2\delta$. We will prove that there is exponential decay at $x''$ in two steps. 

First, there exists an integer $N$ and $\alpha>0$  such that $\Sigma_n(x) \leq n ^ {1-\alpha}$ for all $n \geq N$. For such an integer $n$, integrating $f_n' \geq n^{\alpha} f_n$  between $x'$ and $x$ -- this differential inequality follows from \eqref{eq:mlem}, the monotonicity of the functions $f_n$ (and therefore $\Sigma_n$) and the previous bound on $\Sigma_n(x)$ -- implies that
$$ f_n(x') \leq M\exp(-\delta \,n^\alpha),\quad\forall n\ge N.$$

Second, this implies that there exists $\Sigma < \infty$ such that $\Sigma_n(x') \leq \Sigma$ for all $n$. Integrating $f_n' \geq \tfrac{n}{\Sigma} f_n$ for all $n$ between $x''$ and $x'$ -- this differential inequality is again due to \eqref{eq:mlem}, the monotonicity of $\Sigma_n$, and the bound on $\Sigma_n(x')$ -- leads to
\begin{equation*}f_n (x'') \leq M\exp(-\frac{\delta}{\Sigma} \,n),\quad\forall n\ge0.\label{eq:bb}\end{equation*}

\paragraph{Assume $x>x_1$.} For $n\ge 1$, define the function $T_n := \frac{1}{\log n} \sum_{i=1}^{n} \frac{f_i}{i}$. Differentiating $T_n$ and using \eqref{eq:mlem}, we obtain
$$T_n'~=~ \frac{1}{\log n}\, \sum_{i=1}^{n} \frac{f_i'}{i}  ~\stackrel{\eqref{eq:mlem}}{\geq}~ \frac{1}{\log n}\, \sum_{i=1}^{n} \frac{f_i}{\Sigma_{i}} ~\geq~ \frac{\log \Sigma_{n+1}-\log \Sigma_1}{\log n},$$
where in the last inequality we used that for every $i\ge1$,
$$\frac{f_i}{\Sigma_{i}} \geq \int_{\Sigma_{i}}^{\Sigma_{i+1}} \frac{dt}{t}=\log \Sigma_{i+1}-\log \Sigma_{i}.$$ 
For $x'\in(x_1,x)$, using that $\Sigma_{n+1}\ge\Sigma_n$ is increasing and integrating the previous differential inequality between $x'$ and $x$ gives
$$T_n (x) - T_n(x') \geq (x - x')\,  \frac{\log \Sigma_{n} (x')-\log M}{\log n}.$$
Hence, the fact that $T_n(x)$ converges to $f(x)$ as $n$ tends to infinity implies 
\begin{equation*} f (x) -f(x') 
~\geq~  (x - x')  \, \Big[\limsup_{n\rightarrow \infty} \frac{\log \Sigma_n (x')}{\log n}\Big]~\geq~ x- x'.
\end{equation*}
Letting $x'$ tend to $x_1$ from above, we obtain
$
f(x)\ge x-x_1.
$
\end{proof}

We now present the proof of Theorem~\ref{thm:perco}. Also define
$$\theta_n(p)=\bbP_p[0\longleftrightarrow \partial\Lambda_n]\quad\text{and}\quad S_n:=\sum_{k=0}^{n-1}\theta_k.$$
\begin{lemma}\label{cor:OSSS}
For any $n\ge1$, one has
$$\sum_{e\in E_n} {\rm Inf}_e[\mathbbm 1_{0\leftrightarrow\partial\Lambda_n}]\ge \frac{n}{\displaystyle S_n}\cdot \theta_n(1-\theta_n),$$
where $E_n$ is the set of edges with both endpoints in $\Lambda_n$.
  \end{lemma}

The proof is based on  Theorem~\ref{thm:OSSS} applied to a well chosen algorithm determining the boolean function $\ff:=\mathbbm 1_{0\leftrightarrow\partial\Lambda_n}$. 

One may simply choose the trivial algorithm checking every edge of the box $\Lambda_n$. Unfortunately, the revealment of the algorithm being 1 for every edge, the OSSS inequality will not bring us  interesting information. A slightly better  algorithm would be provided by the algorithm discovering the connected component of the origin ``from the inside''. Edges far from the origin would then be revealed by the algorithm if (and only if) one of their endpoints is connected to the origin. This provides a good bound for the revealment of edges far from the origin, but edges close to the origin are still revealed with large probability. In order to avoid this last fact, we will rather choose a family of algorithms discovering the connected components of $\partial\Lambda_k$ for $1\le k\le n$ and observe that the average of their revealment for a fixed edge will always be small.

\begin{figure}
\begin{center}\includegraphics[width=0.65\textwidth]{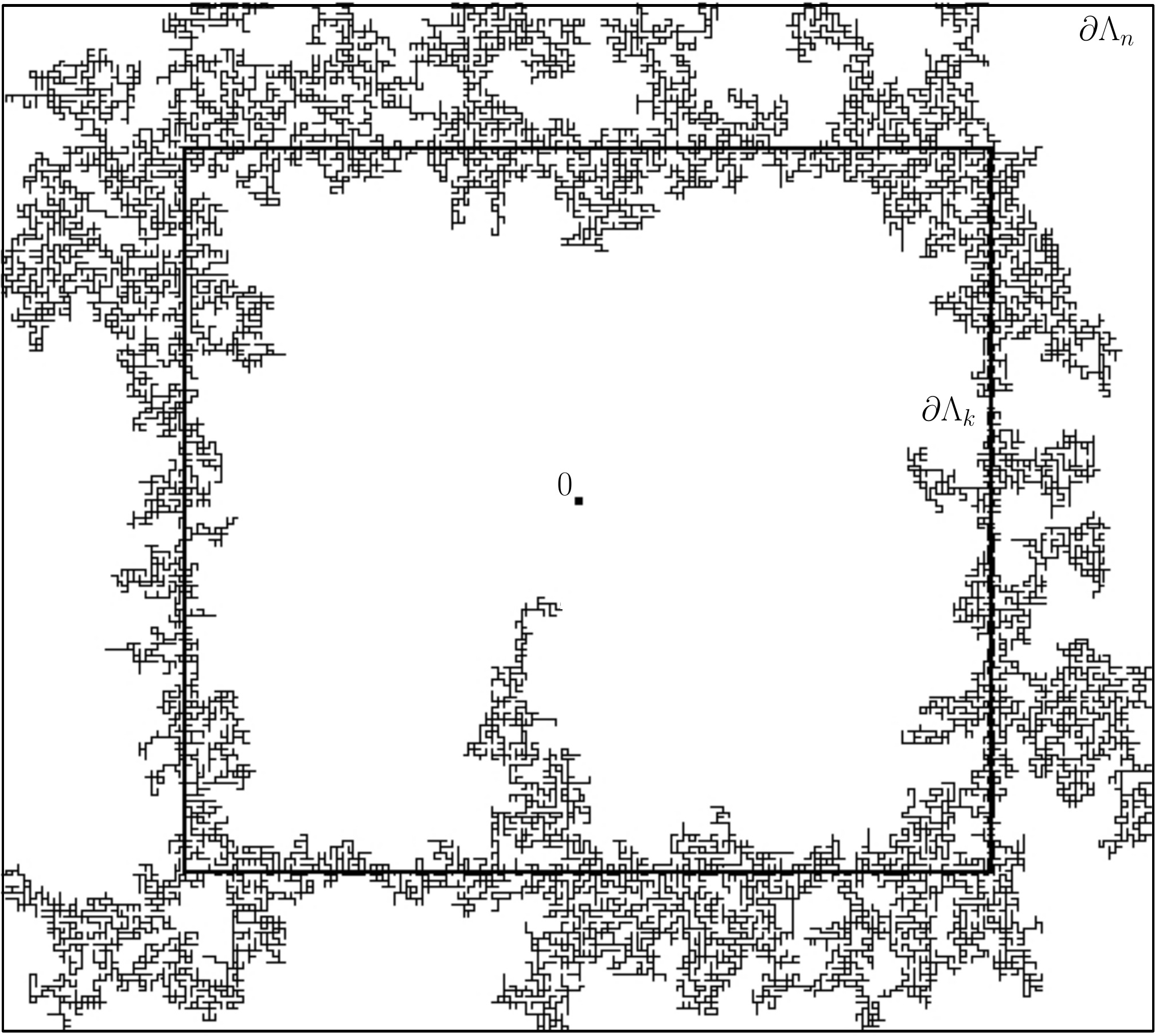}
\caption{\label{fig:algo} A realization of the connected components intersecting $\partial\Lambda_k$. Every edge having one endpoint in this set has been revealed by the algorithm. Furthermore in this specific case, we know that $0$ is not connected to the boundary of $\partial\Lambda_n$. }\end{center}
\end{figure}

\begin{proof} 
For any $k\in[n]$, we wish to construct an algorithm $\TT$ determining $\mathbbm 1 _{0\leftrightarrow\partial\Lambda_n}$ such that for each $e=\{u,v\}$,
\begin{equation}\delta_e(\TT)\le \bbP_p[u\longleftrightarrow \partial\Lambda_k]+\bbP_p[v\longleftrightarrow \partial\Lambda_k].\label{eq:zzz}\end{equation}
Note that this would conclude the proof since we obtain the target inequality by applying Theorem~\ref{thm:OSSS} for each $k$ and then summing on $k$. As a key, we use that for $u\in \Lambda_n$,
\begin{align*}\sum_{k=1}^n\bbP_p[u\longleftrightarrow \partial\Lambda_k]
&~\le~ \sum_{k=1}^n  \bbP_p[u\longleftrightarrow \partial\Lambda_{|k-d(u,0)|}(u)]~\le~ 2S_n.
\end{align*}

We describe the algorithm $\TT$, which corresponds first to an exploration of the connected components in $\Lambda_n$ intersecting $\partial\Lambda_k$ that does not reveal any edge with both endpoints outside these connected components, and then to a simple exploration of the remaining edges.
 
More formally, 
we define $\e$ (instead of the collection of decision rules $\phi_t$) using two  growing sequences  $\partial\Lambda_k=V_0\subset V_1\subset \cdots\subset V$ and $\emptyset=F_0\subset F_1\subset\cdots\subset E_n$ (recall that $E_n$ is the set of edges between two vertices within distance $n$ of the origin) that should be understood as follows: at step $t$, $V_t$ represents the set of vertices that the algorithm found to be connected to $\partial \Lambda_k$, and $F_t$ is the set of explored edges discovered by the algorithm until time $t$.

Fix an ordering of the edges in $E_n$. Set $V_0=\partial\Lambda_k$ and $F_0=\emptyset$. Now, assume that $V_t\subset V$ and $F_t\subset E_n$ have been constructed and distinguish between two cases:\begin{itemize}[noitemsep,nolistsep]
\item If there exists an edge $e=xy\in E_n\setminus F_t$ with $x\in V_t$ and $y\notin V_t$ (if more than one exists, pick the smallest one for the ordering), then set $\mathbf i_{t+1}=e$, $F_{t+1}=F_t\cup\{e\}$ and set $$V_{t+1}:=\begin{cases}V_t\cup\{y\}&\text{ if }\omega_e=1\\ V_t&\text{ otherwise}.\end{cases}$$
\item If $e$ does not exist, set $\mathbf i_{t+1}$ to be the smallest $e\in E_n\setminus F_t$ (for the ordering) and set $V_{t+1}=V_t$ and $F_{t+1}=F_t\cup\{e\}$.\end{itemize}
As long as we are in the first case, we are still discovering the connected components of $\partial\Lambda_k$. Also, as soon as we are in the second case, we remain in it. The fact that $\tau$ is not greater than the last time we are in the first case gives us \eqref{eq:zzz}.

Note that $\tau$ may a priori be strictly smaller than the last time we are in the first case (since the algorithm may discover a path of open edges from 0 to $\partial \Lambda_n$ or a family of closed edges disconnecting the origin from $\partial\Lambda_n$ before discovering the whole connected components of $\partial\Lambda_k$).
\end{proof}

We are now in a position to provide our alternative proof of exponential decay. Fix $n\ge1$. Lemma~\ref{cor:OSSS} together with the Russo-Margulis formula gives
\begin{equation*}\theta_n'=\sum_{e\in E_n}{\rm Inf}_e(\mathbbm 1_{0\leftrightarrow \partial\Lambda_n}) ~\ge~\tfrac{n}{4p(1-p)S_n}\cdot \theta_n(1-\theta_n)\ge \tfrac{n}{S_n} \cdot\theta_n(1-\theta_n).\end{equation*}
Fix $p_0\in(p_c,1)$ and observe that for $p\le p_0$, $1-\theta_n(p)\ge1-\theta_1(p_0)>0$. Then, apply Lemma~\ref{lem:technical} to $f_n=\tfrac1{1-\theta_1(p_0)}\theta_n$. 
Overall, we proved the existence of some $\tilde p_c\in[0,p_0]$ such that 
\begin{enumerate}[noitemsep]
\item\label{item:1}For $p<\tilde p_c$,  there exists $c_p>0$ such that for all $n\ge1$,
$\bbP_p[0\leftrightarrow \partial\Lambda_n]\le \exp(-c_pn)$.
\item\label{item:2} There exists $c>0$ such that for $p>\tilde p_c$, $\bbP_p[0\leftrightarrow \infty]\ge
 c (p-\tilde p_c)$.
\end{enumerate}
Since $p_0$ was chosen larger than $p_c$, $\tilde p_c$ has no choice but to be equal to $p_c$, and the proof of Theorem~\ref{thm:perco} therefore follows.
\section{Generalizations to monotonic measures}

\subsection{Random-cluster model}
Bernoulli percolation is maybe the most classical example of percolation model, but it is far from being the only one. Percolation models appear in various areas of statistical physics as natural models associated with random walks and spin systems. While Bernoulli percolation is a product measure, and the study of random variables in this context boils down to the study of boolean functions on product spaces,  more general percolation models are intrinsically not product measures, and cannot therefore be studied via boolean functions on such spaces. 

Recently, the theory of boolean functions has undergone some progress with the study of monotonic measures. As a consequence, certain results valid for product measures extend to this context, enabling us to apply the proofs of the previous sections to more general percolation models. This discovery led to an explosion of results on these  percolation models, and we propose to discuss some of the progress here. 

Below, we  focus on the {\em random-cluster model}, which is a percolation model introduced by Fortuin and Kasteleyn in \cite{ForKas72} (it is sometimes referred to as the {\em Fortuin-Kasteleyn percolation}) as a unification of different models of statistical physics satisfying series/parallel laws when modifying the underlying graph.
Let $G$ be a finite subgraph of $\bbZ^d$ with vertex-set $V$ and edge-set $E$. Write
$|\omega|=\sum_{e\in E}\omega_e$ and
let $k(\omega)$ denote the number of connected components in the graph $\omega$. 
The probability measure 
$\phi^0_{G,p,q}$ of the random-cluster model on $G$ with {\em edge-weight} 
$p\in[0,1]$, {\em connected component-weight} $q>0$ and free boundary conditions is defined by
\begin{equation}
  \label{probconf}
  \phi_{G,p,q}^0 [\omega] :=
  \frac {p^{|\omega|}(1-p)^{|E|-|\omega|}q^{k(\omega)}}
  {Z_{G,p,q}}
\end{equation}
for every configuration $\omega\in\{0,1\}^{E}$. The constant $Z_{G,p,q}$ is a 
normalizing constant, referred to as the \emph{partition function}, defined in such a way that the sum over all configurations equals 1. For $q\ge1$, the model can be extended to infinite volume by taking the limit as $G$ tends to $\bbZ^d$ of the measures $\phi^0_{G,p,q}$. We denote the infinite-volume measure by $\phi_{\bbZ^d,p,q}^0$.

For $q=1$, the random-cluster model corresponds to Bernoulli percolation.
For integers $q\ge2$, the model is related to Potts models; see below.
For $p\rightarrow 0$ and $q/p\rightarrow0$, the model is connected to electrical networks via Uniform Spanning Trees.

\subsection{Computation of the critical point for random-cluster models on $\bbZ^2$}\label{sec:2.3}

The random-cluster model also undergoes a phase transition at a certain parameter $p_c$ below which the probability that $\omega$ contains an  infinite connected component is zero, and above which this probability is 1. For the square lattice,  the value of this critical point was predicted in physics over forty years ago. Until recently, only the case of $q=1$, $q=2$ and very large values of $q$ were proved. The following theorem finally answers the conjecture for the whole range of parameters $q\ge1$ (the condition $q\ge1$ is not very restrictive, since the behavior of the model is much more tricky for $q<1$ -- for instance some averages of increasing boolean functions are not monotonic in $p$).

\begin{theorem}[Beffara, DC \cite{BefDum12} ]\label{thm:p_c FK}
For the random-cluster model on $\bbZ^2$ with $q\ge1$, the critical point $p_c$ is equal to $\sqrt q/(1+\sqrt q)$.
\end{theorem}
(See also the alternative proofs \cite{DumMan16,DumRaoTas16}.) Note that for $q=1$, we recover the previous theorem with $p_c$ equal to $1/2$.
Similarly, we were recently able to prove the following generalization of Theorem~\ref{thm:perco}.
 \begin{theorem}\label{thm:RCM} Consider the random-cluster model on $\bbZ^d$ with $q\ge1$.
\begin{enumerate}[noitemsep]
\item\label{item:1}For $p<p_c$,  there exists $c_p>0$ such that for all $n\ge1$,
$\phi^0_{\bbZ^d,p,q}[0\leftrightarrow \partial\Lambda_n]\le \exp(-c_pn)$.
\item\label{item:2} There exists $c>0$ such that for $p>p_c$, $\phi^0_{\bbZ^d,p,q}[0\leftrightarrow \infty]\ge
 c (p-p_c)$.
\end{enumerate}
\end{theorem}

The key steps of the proofs of these two theorems are generalizations of Theorems~\ref{thm:RSW}, \ref{maximum_influence} and \ref{thm:OSSS}. 
The generalization of Theorem~\ref{thm:RSW} requires the development of a new RSW theory enabled to tackle percolation models with dependency. This theory led to a number of new applications on these models, including a precise description of the critical behavior (see \cite{DumSidTas13,DumGanHar16,CheDumHon14} and \cite{Dum15,Dum17} for reviews). We chose not to discuss this here, and focus on the generalizations of Theorems~\ref{maximum_influence} and \ref{thm:OSSS}. 

The random-cluster measure satisfies positive association, a property that enables us to study it using probabilistic techniques. In particular, the measure is monotonic in the following sense. A measure $\mu$ on $\{0,1\}^E$ is {\em monotonic} if for any $e\in E$, 
any $F \subset E$, and any $\xi,\zeta\in \{0,1\}^{F}$ satisfying $\xi \le  \zeta $, $\mu[\omega_e=\xi_e,\forall e\in F]>0$ and  $\mu[\omega_e=\zeta_e,\forall e\in F]>0$,   
$$\mu[\omega_e=1 \: |\: \omega_e=\xi_e,\forall e\in F]\le \mu[\omega_e=1 \: |\: \omega_e=\zeta_e,\forall e\in F].$$
As mentioned briefly above, (some of) the theory of boolean functions extends to the case of monotonic measures instead of product measures. In particular, the following two theorems were proved.

\begin{theorem}[Graham, Grimmett \cite{GraGri06}]
There exists a constant $c>0$ such that for any monotonic measure $\mu$ on $[n]$, the following holds. For any increasing boolean function $\ff:\{0,1\}^n\rightarrow\{0,1\}$,
$${\rm Var}_\mu(\ff)\le \frac{c}{\min_i{\rm Var}_\mu[\omega_i]}\,\sum_{i=1}^n \frac{{\rm Inf}_i^\mu[\ff]}{\log (1/{\rm Inf}_i^\mu[\ff])},$$
where ${\rm Inf}_i^\mu(\ff):=\mu[\ff|\omega_i=1]-\mu[\ff|\omega_i=0]$. 
\end{theorem}

\begin{theorem}[DC, Raoufi, Tassion \cite{DumRaoTas17}]
Consider a monotonic measure $\mu$ on $[n]$. Fix an increasing boolean function $\ff:\{0,1\}^n\longrightarrow \{0,1\}$ and an algorithm $\TT$. We have
\begin{equation}
 \mathrm{Var}_\mu(\ff)~\le~   \sum_{i=1}^n  \delta_i(\TT) \, {\rm Inf}_i^\mu(\ff),
  \end{equation}
  where 
$
\delta_i(\TT)
$ is defined as in the product case.
\end{theorem}

The proofs of these statements are not immediate. They are combinations of the strategy for product spaces with encoding of monotonic measures via iid random variables.
 We refer the reader to the corresponding articles \cite{GraGri06} and \cite{DumRaoTas17} for more details. The dependency of the constants $c_\mu$ on $\mu$ is fairly explicit. We do not enter into details but let us say that for random-cluster models with $p$ away from $0$ and $1$, these constants are also away from 0 and infinity.

\subsection{Applications to ferromagnetic lattice spin models}

We conclude these lectures by mentioning one application of the previous results to the study of random colorings of lattice models.
Lattice models have been introduced as discrete models for real life experiments and were later on found useful to model a large variety of phenomena and systems ranging from ferroelectric materials to lattice gas. They also provide discretizations of Euclidean and Quantum Field Theories and are as such important from the point of view of theoretical physics. While the original motivation came from physics, they appeared as extremely complex and rich mathematical objects, whose study required the development of important new tools that found applications in many other domains of mathematics. 

The zoo of lattice models is very diverse: it includes models of spin-glasses, quantum chains, random surfaces, spin systems, percolation models. Here, we focus on a smaller class of lattice models called spin systems. These systems are random collections of spin variables assigned to the vertices of a lattice. The archetypal examples of such models are provided by the Ising model, for which spins take values $\pm 1$, and the Potts model, for which spins take values in a finite set $\{1,\dots,q\}$ representing colors (note that the Ising model corresponds to the Potts model with $q=2$, where $+1$ and $-1$ are identified with the two colors 1 and 2). We refer to \cite{Gri06} for more details.

The random-cluster model is related to the Potts models via a simple coupling: one obtains the Potts models by coloring  each connected component of the random-cluster configuration $\omega$ uniformly at random\footnote{Meaning that for each connected component $\calC$, one chooses a color $\sigma_\calC$ uniformly and independently of the choices for other connected components, and then one assigns to every vertex $x$ of $\bbZ^d$ the color $\sigma_x$ equal to $\sigma_\calC$ where $\calC$ is the unique connected component of $\omega$ containing $x$. Note that automatically, all the vertices in the same connected component receive the same color.}. As a consequence, one can obtain new results on these models, such as rigorous computations of the so-called critical inverse temperatures separating the disordered phase from the ordered phase, the exponential decay of correlations in the disordered phase, the continuity/discontinuity of the phase transition in two dimensions, etc.

There are many models of statistical physics, and therefore many potential applications of dependent percolation models. In particular, let us mention that the techniques described in these notes were also very useful to study continuous percolation models. We refer to \cite{DumRaoTas17a,DumRaoTas17b} for two typical examples. In conclusion, the use of abstract sharp threshold inequalities, which are not necessarily intuitive from the point of view of physics, will probably generalize in the next few years, and we expect a number of breakthroughs in the field based on similar ideas.

\bibliographystyle{plain}

\end{document}